\documentclass{amsart}
\usepackage{amssymb, latexsym} 
\newtheorem{theorem}{Theorem}
\newtheorem{lemma}{Lemma}
\newtheorem{remark}{Remark}
\newtheorem{corollary}{Corollary}
  
\begin{document}

\title[A Generalization of a Power-Conjugacy Problem]{A Generalization of a Power-Conjugacy Problem in Torsion-Free Negatively Curved Groups}
\author{Rita Gitik}
\address{ Department of Mathematics \\ University of Michigan \\ Ann Arbor, MI, 48109}
\email{ritagtk@umich.edu}
\date{\today}

\begin{abstract}
{Let $H$ and $K$ be quasiconvex subgroups of a negatively curved torsion-free group $G$. 
We give an algorithm which decides whether an element of $H$ is conjugate in $G$ to an element of $K$.}
\end{abstract}


\maketitle

\section{Introduction}

In 1911 Max Dehn introduced  in \cite{De} three basic algorithmic problems in
group theory: the word problem, the conjugacy problem, and the isomorphism problem. 
Let a group $G$ be given by a presentation $G=<X|R>$.
The word problem asks if there exists an algorithm to decide if any word 
in the alphabet $X$ represents the trivial element of $G$. The word problem was shown to be undecidable, in general, 
by Novikov, \cite{No}, and independently, by Boone, \cite{Bo}. The conjugacy problem
asks if there exists an algorithm which for any pair of words in the alphabet $X$  decides whether 
they represent conjugate elements in $G$. A special case of the conjugacy problem, namely the existence of an algorithm deciding if
a given word in the alphabet $X$ represents an element of $G$ conjugate to the identity of $G$, is the word problem. Hence the conjugacy
problem is also undecidable, in general. 
The isomorphism problem asks if for any pair of presentations there exists 
an algorithm to decide if they define isomorphic groups. The isomorphism problem was shown to be undecidable, in general, by Adian, \cite{Ad}, 
and independently by Rabin, \cite{Ra}.
The membership problem for a subgroup $H$ of a group $G$ asks if there exists an algorithm deciding if any element of $G$ belongs to $H$.
As the word problem, in general, is undecidable, it follows that the membership problem is, in general, undecidable.
The power-conjugacy problem for a group $G$ asks if for any two elements of $G$ there exists an algorithm to decide 
if one of them is conjugate to some power of the other. A special case of the power-conjugacy problem, namely the existence of an algorithm deciding if any element of $G$ is conjugate to some power of the identity element in $G$, is the word problem. Hence the power-conjugacy problem is undecidable, in general. For more detailed information about the aforementioned algorithmic problems see, for example, survey articles  \cite{Gi3}, \cite{Hu}, and \cite{Mi}.

Even though the aforementioned  algorithmic problems are undecidable in general, they are decidable in 
negatively curved groups. The solution of the word problem in negatively curved groups follows from the work of Greendlinger, \cite{Gre}. 
The solution of the conjugacy problem  for negatively curved groups was given by Gromov in \cite{Gro} p.199. The solution of the 
isomorphism problem for negatively curved groups was given by Dahmani and Guirardel in \cite{D-G}. The solution of the membership problem for quasiconvex subgroups of negatively curved groups was given by the author, \cite{Gi1}, \cite{Gi2}, and \cite{Gi3}, and independently, by Farb, \cite{Fa}, I. Kapovich, \cite{Ka}, and Kharlampovich, Miasnikov, and Weil, \cite{K-M-W}. The power-conjugacy problem was shown to be decidable when $G$ is negatively curved by Lysenok, \cite{Ly}.

The power-conjugacy problem has been the subject of extensive research and was solved for several additional classes of groups, see for example,
\cite{A-S}, \cite{Be}, \cite{B-D-R}, \cite{B-K}, \cite{B-M}, \cite{Co}, \cite{L-M}, and \cite{Pr}.
In this paper we prove a generalized version of the power-conjugacy problem for torsion-free negatively curved groups.

\begin{theorem}
Let $H$  and $K$ be $\mu$-quasiconvex subgroups 
of a $\delta$-negatively curved torsion-free group $G$.
There exists an algorithm to decide if an element 
of $H$ is conjugate in $G$ to an element of $K$.
\end{theorem}

\begin{corollary} 
Let $K$ be a quasiconvex subgroup of a
torsion-free negatively curved group $G$,
and let $u$
be a non-trivial element of $G$. There exists an algorithm to decide 
whether some power of $u$ is
conjugate in $G$ to an element of $K$.
\end{corollary}
\begin{proof}
As a cyclic subgroup in a negatively curved group is quasiconvex, 
\cite{Gro} p. 210, we can apply Theorem 1 with $H$ being the cyclic
subgroup generated by $u$.
\end{proof}

\begin{corollary} 
Let $K$ be a quasiconvex subgroup of a
torsion-free negatively curved group $G$ and let $u$
be a non-trivial element of $G$. There exists an algorithm to decide 
whether  $u$ is conjugate in $G$ to an element of $K$.
\end{corollary}
\begin{proof}
Lemma 1, stated below, shows that if $H$ is the cyclic
subgroup generated by $u$ and $u$ is conjugate
to an element of $K$, then there exists $g \in G$ with $|g| <C$ such that
$g u g^{-1} \in K$, ($C$ is defined in the statement of Lemma 1).
As $G$ is finitely generated, there are only finitely many elements shorter than $C$ in $G$.
Hence we need to check if one of finitely many 
elements of the form $g u g^{-1}$ with  $|g| <C$ is in $K$, which we can do because the
membership problem for $K$ in $G$ is decidable.
\end{proof}

\begin{corollary}
The power-conjugacy problem is decidable for torsion-free negatively curved groups.
\end{corollary}
\begin{proof}
Let $u$ be a non-trivial element of $G$ and let $v$ be any element of $G$.
Corollary 2 implies that there is an algorithm to decide whether $u$ is conjugate in $G$
to an element of a cyclic group generated by $v$, which is the power-conjugacy problem.
\end{proof}

Theorem 1 follows from two  technical results stated below.

\begin{lemma} 
Let  $H$  and $K$ be $\mu$-quasiconvex subgroups 
of a $\delta$-negatively curved torsion-free group $G$, 
and let $g \in G$ be a shortest representative of 
the double coset $KgH$ such that $ghg^{-1}$ is in $K$ 
for some non-trivial element $h$ of  $H$. 
Then $g$ is shorter than $C=4\delta +2\mu +(m^2 +1) \cdot L$, where
$L$ is the number of words in $G$ with 
length less than $ 8 \delta +\mu$, 
and $m$ is the number of elements in $G$ with length not greater
than $42 \delta +12 \mu$.
\end{lemma}

\begin{lemma}
Let  $H$  and $K$ be $\mu$-quasiconvex subgroups 
of a $\delta$-negatively curved torsion-free group $G$ and let $h$ be a 
shortest non-trivial element of  $H$ such that $ghg^{-1}$ is in $K$ 
for some $g \in G$ with $|g| <C$. Then $h$ is shorter than 
$C'= (L'+2)2\mu+8 \delta$, where $L'$ is the number of words in $G$ shorter than $(2 \delta +2\mu)$.
\end{lemma}

\begin{remark}
Note that if there exist $h \in H$ and $g \in G$ such that $ghg^{-1} \in K$, 
then for any $h_0 \in H$ and $k_0 \in K$, 
$(k_0 g h_0) (h_0^{-1} h h_0) (h_0^{-1}g^{-1} k_0^{-1}) \in K$.
So if $g \in G$ conjugates an element of $H$ to an element of $K$, 
then any $g_0$ in the double coset $KgH$ has the same property. 
\end{remark}

\textbf{Proof of Theorem 1.}

Assume that there exists a non-trivial element $h \in H$ and an element
$g \in G$ such that $ghg^{-1} \in K$. Let $g_1$ be a shortest element in the double coset $KgH$.
Lemma 1 states that $|g_1| <C$.
Remark 1 implies that there exists an element $h' \in H$ such that $g_1h'g_1^{-1} \in K$. 
Let $h_1 \in H$ be a shortest non-trivial element such that  
$g_1h_1g_1^{-1} \in K$.  Lemma 2 states that $|h_1| < C'$.
As $G$ is finitely generated, there are finitely many possible $g_1$ and 
$h_1$. Hence we need to form finitely many products  $g_1h_1g_1^{-1}$ 
and to check if they belong to $K$, which we can verify because the 
generalized word
problem is solvable for quasiconvex subgroups of negatively curved
groups.

\section{Preliminaries}

Let $X$  be a set, let $X^* = \{x,x^{-1} |x \in X \}$, and for  
$x \in X$ define $(x^{-1})^{-1} =x$. 
A word in $X$  is any finite sequence of elements of $X^*$.
Denote the set of all words in $X$ by W(X), and denote the equality
of two words by " $ \equiv $".

Recall that the Cayley graph of $G= \langle X|R \rangle $, denoted $Cayley(G)$,
is an oriented graph whose set of vertices is $G$ and the set of edges is 
$G \times X^*$, such that   an edge $(g,x)$ begins at the vertex $g$ and ends at the 
vertex $g x $. Since the Cayley graph  depends on the generating set of the group, we work with 
a fixed generating set.

A geodesic in the Cayley graph is a shortest path joining two vertices.
A geodesic triangle in  the Cayley graph is a closed path
$p=p_1p_2p_3$, where each $p_i$ is a geodesic.
A group $G=\langle X|R \rangle$ is $\delta$-negatively curved if any side of any
geodesic triangle in  the Cayley graph of $G=\langle X|R \rangle $ belongs to the 
$\delta$-neighborhood of the union of the other two sides.

A subgroup $H$ of a group $G= \langle X|R \rangle $ is $\mu$-quasiconvex  in 
$G= \langle X|R \rangle $ if any geodesic in the Cayley graph of 
$G= \langle X|R \rangle $ with endpoints in $H$ belongs to the $\mu$-neighborhood of $H$. 
A subgroup is quasiconvex in $G= \langle X|R \rangle $ if it is $\mu$-quasiconvex in 
$G= \langle X|R \rangle $ for some $\mu$. 
As usual, we assume that all negatively curved groups are finitely generated.

The label of a path   
$p = (g,x_1) (g \cdot x_1, x_2) \cdots (g \cdot x_1 \cdots x_{n-1},x_n)$
in $Cayley(G)$ is the function $Lab (p) \equiv x_1 x_2 \dots x_n \in W(X)$.
As usual, we identify the word $Lab(p)$ with the corresponding element in $G$.

\bigskip

The following result has been proven in \cite{G-M-R-S}.

\textbf{Theorem GMRS}

\emph{Let $H$ be a $\mu$-quasiconvex subgroup of a $\delta$-negatively curved torsion-free group $G$, 
and let $m$ be the number of elements in $G$ with length not greater
than $42 \delta +12 \mu$. Let $S= \{ g_i^{-1} H g_i |1 \le i \le n \}$
be a collection of  essentially distinct conjugates of $H$, where
the conjugates $g_i^{-1} H g_i$ and $g_j^{-1} H g_j$ are
called essentially distinct if $Hg_i \neq Hg_j$ for $i \neq j$. 
If $n > m^2$, then  the intersection of some pair of elements of $S$ is
trivial.} 

\section{Proofs of the Results}

Let $g$ be a shortest element in the double 
coset $KgH$ such that   $ghg^{-1}=k$ is in $K$ for some non-trivial 
element $h$ of $H$. 

Let $p, p_h$ and  $p'$ be geodesics in $Cayley(G)$ such that 
$Lab( p) \equiv Lab(p') \equiv g, Lab( p_h) =h$, $p$ begins 
at $1$ and ends at $g$, $p'$ begins at $ghg^{-1}$ and ends at $gh$, and $p_h$ 
begins at $g$ which is the endpoint of $p$ and ends at $gh$ which is 
the endpoint of $p'$. 

Denote the vertices of $p$ in their linear order by $1=v_0,v_1, \cdots ,
v_n=g$ and denote the vertices of $p'$ in their linear order by $ghg^{-1} =v'_0,v'_1, \cdots , 
v'_n=gh$. Note that  $|g|=|p|=|p'|=n$.

Let $p_k$ be a geodesic in $Cayley(G)$ joining  $v_0=1$ and $v'_0=ghg^{-1}$.
Then the paths $p, p_k, p' $ and  $p_h$ form a geodesic $4$-gon 
which is $2 \delta$-thin in $Cayley(G)$, because $G$ is 
$\delta$-negatively curved.

\begin{lemma}
For any index $i$ such that $ 2 \delta +\mu \le i \le 
n-2 \delta - \mu$ the distance 
$d(v_i,v'_i)$ is less than $8 \delta + \mu$.
\end{lemma}
\begin{proof}
Let  $l$ be the biggest index such that
the vertex   $v_l$ belongs to the $2 \delta$-neighborhood of $p_k$, 
let $w_l$ be a vertex in $p_k$ closest to $v_l$, and let $r$ be a geodesic
joining $w_l$ to $v_l$. By construction, $Lab(p_k) =k \in K$. As $K$ is $\mu$-quasiconvex,  $p_k$
belongs to the $\mu$-neighborhood of $K$ in $Cayley(G)$, hence there exists
a vertex $u_l \in K$ such that $d(w_l,u_l) < \mu$. Let $r'$ be a geodesic
joining $u_l$ to $w_l$. Let $s_l$ be a subpath of $p$ joining
$v_0$ to $v_l$, let $\bar{s_l}$ be the inverse of the path $s_l$, and let $t_l$ be a subpath of $p$ joining $v_l$ to $v_n$.
Note that $Lab(r'rt_l)=Lab(r'r\bar{s_l})(s_lt_l)=Lab(r'r\bar{s_l})g \in Kg$.
As $g$ is a shortest representative of $K g H$, it follows that
$|g|=|p|=|s_l|+|t_l| \le |r' r t_l| < 2\delta+ \mu +|t_l|$, so
$|s_l|= d(v_0, v_l)=l < 2 \delta +\mu$. 
Hence if $i \ge \mu+2 \delta$, then $d(v_i,p_k) > 2\delta$.

Let  $i$ be the smallest index such that
the vertex   $v_i$ belongs to the $2 \delta$-neighborhood of $p_h$.
An argument, similar to the above, shows that 
for any   $j \le n -  2 \delta - \mu$, $d(v_j,p_h)> 2 \delta$. 

Therefore, for any index $i$ such that 
$2 \delta +\mu \le i  \le n-2 \delta -\mu$, the 
vertex   $v_i$ belongs to the $2 \delta$-neighborhood of $p'$.
Similarly,  for any index $i$ such that 
$2 \delta +\mu \le i \le n-2 \delta -\mu$
the vertex $v'_i$ belongs to the $2 \delta$-neighborhood of $p$.

Let $b=n-2\delta-\mu$. We claim that $d(v_b, v'_b) < 4\delta+\mu$.
Indeed, let $j(b) \le b$ be an index such that 
$d(v_b,v'_{j(b)}) < 2 \delta$.
Let $t_b$ be the subpath of $p$ joining $v_b$ and $v_n$, let
$t'_{j(b)}$ be the subpath of $p'$ joining $v'_{j(b)}$ to $v'_n$, and let
$\gamma$ be a geodesic joining $v_b$ and $v'_{j(b)}$.
Consider the geodesic $4$-gon formed by $t_b, p_h, t'_{j(b)}$ and $\gamma$.

As $ b \le n-2 \delta -\mu$, it follows that 
$d(v'_b, p_h) > 2 \delta$.  If $d(v'_b, \gamma) < 2\delta$,
then $d(v_b, v'_b) \le |\gamma| + d(v'_b, \gamma) < 4 \delta$.
If $d(v'_b, t_b) < 2\delta$,
then $d(v_b, v'_b) \le |t_b| + d(v'_b, t_b) < 4 \delta +\mu$.

Now consider $ 2 \delta +\mu \le i \le  n- 2\delta -\mu$.
Let $j(i)$ be an index such that $d(v_i,v'_{j(i)}) < 2 \delta$.
By interchanging
$v_i$ and $v'_{j(i)}$, if needed, we can assume that $j(i) \ge i$. As $p$ 
is a geodesic, $d(v_i, v_b)=b-i \le d(v_i,v'_{j(i)})+d(v'_{j(i)},v'_b)+
d(v_b, v'_b) < 2 \delta +(b-j(i))+4\delta +\mu$, hence 
$0 \le j(i)-i  < 6 \delta +\mu$.
But then $d(v_i,v'_i) \le d(v_i,v'_{j(i)})+d(v'_{j(i)},v'_i) < 2
\delta +(j(i)-i)  < 8 \delta +\mu$, proving Lemma 3.
\end{proof}

\textbf{Proof of Lemma 1.}

Assume that  $|g|=n \ge C$, where $C$ is defined in the statement of Lemma 1. It follows that 
$(n-2 \delta -\mu)- (2\delta +\mu) \ge C-4 \delta -2\mu = L \cdot (m^2 +1)$. 
Therefore Lemma 3 implies that  there exists a set of distinct indexes
$\{ i_j | 1 \le j \le m^2 +1 \}$ such that: 

\begin{enumerate}
\item $n-2\delta-\mu \ge i_j \ge 2\delta +\mu$,

\item the paths  connecting $v_{i_j}$ to $v'_{i_j}$ 
 have the same label, say $a$, for all $i_j$.
\end{enumerate}

Let $s_{i_j}$ be the initial subpath of $p$ connecting $v_0$ and $v_{i_j}$ 
and let $s'_{i_j}$ be the initial subpath of $p'$ connecting $v'_0$ and
$v'_{i_j}$.
 If $a=1$, then $v_{i_1}=v'_{i_1}$. 
It follows that $Lab(s_{i_1})^{-1}  k Lab(s'_{i_1}) =1$, 
 hence $k=h=1$, contradicting the choice of $h$. 

If $a \neq 1$, consider the set 
$S= \{Lab(s_{i_j}^{-1} )k Lab(s_{i_j}), 1 \le i_j \le m^2 +1 \}$.
As $Lab(s_{i_j})^{-1}k Lab(s_{i_j}) = a \neq 1$ for all 
$1 \le i_j \le m^2+1$, and as $G$ is 
torsion-free, it follows
that the intersection of any pair of  elements of $S$ is infinite.

However, the elements of $S$ are essentially distinct.
Indeed, assume that there exists $k_0 \in K$ such that 
$k_0 Lab(s_{i_j})=Lab(s_{i_l})$.

Without loss of generality, $i_l > i_j$.  Let $t_{i_l}$ be the subpath of $p$ joining $v_{i_l}$ to 
$v_n$. Then $g=Lab(s_{i_l}) Lab(t_{i_l})=k_0 Lab(s_{i_j}) Lab(t_{i_l})$. 
Hence the element $Lab(s_{i_j}) Lab(t_{i_l})$
belongs to $Kg$ and 
$|Lab(s_{i_j}) Lab(t_{i_l})| \le |s_{i_j}|+|t_{i_j}| <|s_{i_l}| + |t_{i_l}|=|g|$,
contradicting the choice of $g$ as a shortest representative of the double coset $KgH$.
So $S$ is a collection of $m^2+1$ distinct conjugates of $K$ such that 
any two elements of $S$ have infinite intersection, contradicting Theorem GMRS.  

Hence $|g| < C$, proving  Lemma 1.

\begin{remark}
By increasing the quasiconvexity constant $\mu$ if needed,
we can assume that $\mu$ is a positive integer.
\end{remark}

\begin{lemma} 
Let $g$ be an element shorter than $4\delta + 2\mu$ such that 
$ghg^{-1} \in K$
for a non-trivial $h \in H$. If $h$ is longer than 
$(L'+2)2\mu +8 \delta$, where $L'$ is the number of words in $G$
shorter than $2 \delta +2\mu$, 
then there exist a non-trivial $h_0 \in H$ with $|h_0| \le 2\mu (L'+2)$ 
and  $g_0 \in G$ with $|g_0| < 2\delta + 2\mu$ 
such that $g_0 h_0 g_0^{-1} \in K$.
\end{lemma}
\begin{proof}
Let $p, p_k, p' $ and  $p_h$ be a geodesic $4$-gon, as in the proof of 
Lemma 3. Denote the vertices of $p_h$ in their linear order by $g=v_0^h, v_1^h, \cdots v_f^h=gh$.

Let $q'$ be the maximal initial subpath of $p_h$
which belongs to the $2 \delta$-neighborhood of $p$. Note that the length of $q'$ is at most $4\delta +\mu$.
Indeed, let $v^h_{q'}$ be the terminal vertex of $q'$. Let $v_{q'}$ be a vertex of $p$ such that $d(v^h_{q'}, v_{q'}) \leq 2 \delta$.
Let $\alpha$ be a geodesic in $Cayley(G)$ which begins at $v_{q'}$ and ends at $v^h_{q'}$. Let $s_{q'}$ be the initial subpath of $p$
joining $v_0$ to $v_{q'}$ and let $t_{q'}$ be the terminal subpath of $p$ joining $v_{q'}$ to $v_n=g$. As $H$ is $\mu$-quasiconvex in $G$,
there exists a vertex $x_{q'}$ in $Cayley(G)$ and a geodesic $\alpha'$ joining $v^h_{q'}$ to $x_{q'}$ such that $Lab(q_h \alpha') \in H$ and
$|\alpha'| < \mu$. As $g$ is a shortest element in the double coset $KgH$, it follows that 
$|g|=|s_{q'}|+|t_{q'}| \leq |s_{q'}|+|\alpha|+|\alpha'| \leq |s_{q'}|+ 2\delta + \mu$. Hence $|t_{q'}| \leq 2\delta + \mu$.
It follows that $|q'| \leq |t_{q'}| + |\alpha| \leq 4\delta + \mu$.

Similarly, the length of the maximal subpath of $p_h$
which belongs to the $2 \delta$-neighborhood of $p'$ is at most $4\delta +\mu$.

Assume that $h$ is longer than $(L'+2)2\mu+8 \delta$. 
Then there exists a subpath $q$ of $p_h$ of length at least
$(L'+1) 2\mu $ which belongs to the $2 \delta$-neighborhood of $p_k$. 
By construction, $q$ begins at the vertex $v^h_{q'}$.
By definition of the path $q$, for any vertex
$v^h_i$ of $q$ there exists a vertex $w(v^h_i)$ in $p_k$ such that $d(v^h_i, w(v^h_i)) < 2\delta$. 
As $H$ is $\mu$-quasiconvex in $G$,  for any vertex
$v^h_i$ of $q$ there exists a vertex $x_i$ such that $d(v^h_i, x_i) <\mu$, and the 
element $x_i$ belongs to the coset $gH$. Similarly, there exists a vertex $k(v^h_i)$ such that 
$d(w(v^h_i), k(v^h_i)) <\mu$ and the element $k(v^h_i)$ belongs to $K$.  Let
$\beta_i$ be a geodesic joining $k(v^h_i)$ and $x_i$. 
Then $|\beta_i| < 2\mu +2 \delta$.

Consider a subset of vertices of $q$ with indexes $v^h_{q'}, v^h_{q' + 2 \mu}, \cdots, v^h_{q' + j \cdot 2 \mu}, \cdots v^h_{q' + L' \cdot 2 \mu}$.
The distance between two consecutive vertices in this subset is  $2\mu$, hence $x_{q' +i \cdot 2\mu} \neq x_{q' +j \cdot 2\mu}$ for $i \neq j$. 

By definition of the constant $L'$, there exist indexes $i \neq j$  such that $Lab(\beta_{q' + i \cdot 2\mu})=Lab(\beta_{q'+j \cdot 2 \mu})$.
By construction, $d(v_{q'+i \cdot 2 \mu}, v_{q'+j \cdot 2 \mu} ) \le 2\mu(L' +1)$, 

so $d(x_{q'+i \cdot 2\mu}, x_{q'+j \cdot 2\mu}) < 2\mu+2\mu(L'+1)=2\mu(L'+2)$. 

By construction, if $\nu$ is a geodesic joining $x_{q'+i \cdot 2\mu}$ and  $x_{q'+j \cdot 2\mu}$,  then
$Lab(\nu) \in  H$. Similarly, if $\nu'$ is a geodesic joining $k(v_{q'+i \cdot 2\mu})$ and $k(v_{q'+j \cdot 2\mu})$,  then
$Lab(\nu') \in  K$. 
So take $g_0=Lab(\beta(v_{q'+i \cdot 2\mu})$ and $h_0=Lab(\nu)$, proving Lemma 4.
\end{proof}

\textbf{Proof of Lemma 2.}

Let $h$ be a non-trivial element of $H$ such that $ghg^{-1} \in K$ for 
some $g \in G$ with $|g| <C$, where $C$ is defined in the statement of Lemma 1. 
We want to find  $h_0 \in H$ with $|h_0| < C'$, where $C'$ is defined in the statement of Lemma 2,
and $g_0 \in G$, which might be different from $g$, with $|g_0|< C$ such that   $g_0h_0g_0^{-1} \in K$.

Consider three cases.
\begin{enumerate}
\item
If $|g| < 4\delta + 2\mu$ and $|h| \le (L'+2)2\mu+8 \delta$, 
take $g_0=g$ and $h_0=h$.
\item 
If $|g| < 4\delta + 2\mu$ and $|h| > (L'+2)2\mu +8 \delta$, then Lemma 4 states that
there exist a non-trivial $h_0 \in H$ with $|h_0| \le 2\mu (L'+2)$ 
and  $g_0 \in G$ with $|g_0| < 2\delta + 2\mu$ 
such that $g_0 h_0 g_0^{-1} \in K$.
\item
If $C > |g| \ge 4\delta + 2\mu$, let $p, p', v_b, v'_b$ and $p_h$ be
as in the proof of Lemma 3. 
It is shown in Lemma 3 that $d(v_b, v'_b) < 4\delta +\mu$. 
Then $|h|=|p_h| \le d(v_b,v_n) +d(v_b, v'_b)
+ d(v'_b,v'_n) < (\mu+ 2\delta) + (\mu+ 4\delta) + (\mu+ 2\delta) 
<(3\mu+ 8\delta)$.
Hence we can take $g_0=g$ and $h_0=h$, proving Lemma 2.
\end{enumerate}

\section{Acknowledgment}
The author would like to thank Hans Boden, Mike Davis, Cameron Gordon, Dinakar Ramakrishnan, and Eliyahu Rips for their support.

\enddocument